\documentclass[letterpaper, 10 pt, conference]{ieeeconf}  

\IEEEoverridecommandlockouts                              

\usepackage{hyperref}
\usepackage{amsmath}
\usepackage{amssymb}

\usepackage{mathtools}
\usepackage{float}
\usepackage{graphicx}
\usepackage{setspace}
\usepackage[dvipsnames]{xcolor}
\usepackage{caption}
\usepackage{subcaption}
\usepackage{braket}
\usepackage{comment}

\usepackage{algorithm}
\usepackage{algpseudocode}
\usepackage{cite}

\newtheorem{theorem}{Theorem} 
  
 \newtheorem{proposition}{Proposition} 
 \newtheorem{corollary}{Corollary}[section] 
 \newtheorem{definition}{Definition}[section]
  
 \newtheorem{remark}{Remark}[section] 
 \newtheorem{lemma}{Lemma} 
 \newtheorem{example}{Example}[section]

 \usepackage{bm} 
\usepackage{multirow}
\usepackage{adjustbox}

\usepackage{graphicx}      

\newcommand{\be}{\begin{equation}}
\newcommand{\ee}{\end{equation}}

\newcommand{\bma}{\begin{bmatrix}}
\newcommand{\ema}{\end{bmatrix}}
\newcommand{\R}{\mathbb{R}}

\newcommand{\mc}{\mathcal}

\title{\LARGE \bf
Interpolation Conditions for Instant Data Consistency with Port-Hamiltonian Structure
}

\author{Martina Vanelli, Nima Monshizadeh, Julien M. Hendrickx
	\thanks{M. Vanelli and N. Monshizadeh are with the Engineering and Technology Institute
		Groningen, University of Groningen, Nijenborgh 4, 9747 AG
		Groningen, The Netherlands. Email: 
		{\tt \footnotesize \{m..vanelli, n.monshizadeh\}@rug.nl.} 
		J. M. Hendrickx is with the ICTEAM Institute, UCLouvain, B-1348, Louvain-la-Neuve, Belgium	
	(email:	{\tt \footnotesize	julien.hendrickx@uclouvain.be}).  This work was supported by the Concerted Research Action (ARC) via the
“SIDDARTA” project 
and by FNRS via the "InterpoControl" Research Project.}
}

\begin{document}

\maketitle
\thispagestyle{empty}
\pagestyle{empty}

{
\begin{abstract}
We develop a data-driven framework for nonlinear port-Hamiltonian (pH) systems based on interpolation conditions to characterize consistency between observed data and structured dynamical models. Specifically, we derive necessary and sufficient conditions for the existence of a pH system with a smooth (convex) Hamiltonian instantly consistent with a given dataset, without requiring explicit parametrization. We further provide a semidefinite programming formulation to verify consistency with non-degenerate interconnection and dissipation structures. Our results provide a principled approach to assess instant data consistency with physical structure and pave the way for control design directly from data.

\end{abstract}

\section{Introduction}
Data-driven control has emerged as a powerful paradigm for the analysis and control of dynamical systems \cite{markovsky2021behavioral}, enabling controller design directly from measured data without explicit model identification. Closely related to the behavioral approach, it has achieved strong guarantees for linear systems, including stabilization, optimal control, and robustness \cite{de2019formulas, van2020willems, berberich2020data, bisoffi2022data}, as well as hybrid approaches that combine data with partial prior knowledge \cite{berberich2022combining, niknejad2023physics}.

Extending these ideas to nonlinear systems remains challenging due to their intrinsic complexity and richer dynamical behavior. Existing approaches include virtual reference control \cite{campi2002vrft}, kernel-based methods \cite{tanaskovic2017data}, model-free and intelligent PID approaches \cite{fliess2013modelfree}, sampled-data control schemes \cite{fraile2020stabilization}, nonlinear data-enabled predictive control \cite{berberich2024overview}, and operator-theoretic techniques such as dynamic mode decomposition and Koopman-based design \cite{proctor2016dmdc,korda2018koopman}. Despite these advances, many methods remain tailored to specific system classes, such as bilinear \cite{bisoffi2020bilinear}, polynomial \cite{guo2021polynomial}, rational \cite{strasser2021beyond}, or flat systems \cite{alsalti2021flat}. Even within more unified viewpoints \cite{monshizadeh2025versatile}, existing approaches typically rely on restrictive modeling choices, such as predefined basis functions, lifting maps, or parametrized function classes, limiting their applicability.

In this work, we adopt a different perspective by integrating measured data with prior physical insight without requiring explicit parametrization of the system dynamics. Specifically, we investigate whether observed trajectories are \emph{instantly consistent} with a structured class of dynamical systems, and whether this consistency can serve as a foundation for control design.

Building on previous work for linear systems \cite{vanelli2025interpolation}, we extend this viewpoint to nonlinear dynamics by focusing on port-Hamiltonian (pH) systems. This class is particularly relevant for physical applications, as it captures key properties such as energy conservation, dissipation, and interconnection structure \cite{van2014port}. The pH framework represents systems through an energy function together with structure and dissipation matrices, naturally encoding passivity. These features make pH systems well suited for control design \cite{van2000l2}, including passivity-based control and interconnection and damping assignment (IDA-PBC) techniques \cite{ortega2002interconnection}. While pH systems have also been studied in data-driven control, including adaptive control \cite{Wang2007,Dirksz2010}, reinforcement learning \cite{Sprangers2014}, iterative and repetitive learning control \cite{Fujimoto2003}, robust design \cite{Ryalat2018}, and Bayesian approaches \cite{BeckersGPPHS2022,BeckersCDC2023,BeckersColombo2025}, these methods typically rely on parametrized Hamiltonians, fixed or parametrized interconnection and damping structures, or chosen basis functions.

Within this setting, we derive necessary and sufficient conditions for the existence of a port-Hamiltonian system, with prescribed properties such as smoothness and convexity of the Hamiltonian, that is instantly consistent with a given dataset. The proposed conditions depend solely on the data and avoid explicit parametrization, which is particularly advantageous in nonlinear settings where suitable parameterizations may be unavailable. Moreover, the framework does not impose requirements on dataset size, instead combining data with prior physical structure to characterize what can be inferred directly from the available information.

Our approach relies on interpolation conditions, i.e., necessary and sufficient conditions on a finite set of points that guarantee the existence of a function within a prescribed class interpolating the data. Such conditions have recently played a central role in optimization, particularly for smooth convex functions \cite{taylor2017smooth,taylor2017convex}, enabling tight worst-case performance analysis \cite{rubbens2023interpolation,bousselmi2024interpolation}. Here, we take a first step toward leveraging these tools in data-driven control to characterize trajectories consistent with structured dynamical systems.

The main contributions of this paper are as follows:
\begin{itemize}
\item We introduce an interpolation-based framework to characterize instant data consistency with port-Hamiltonian systems, without requiring any parametrization of the Hamiltonian or system dynamics.
\item We derive necessary and sufficient algebraic conditions for consistency with smooth and smooth convex Hamiltonians, expressed purely in terms of data.
\item We show that, in the full-rank case, these conditions admit an equivalent convex semidefinite programming (SDP) formulation, enabling tractable verification.
\end{itemize}
We illustrate the proposed approach through numerical simulations and conclude with a discussion on how the framework could be extended toward data-driven control design. 

The rest of the paper is organized as follows. We recall the preliminaries on port-Hamiltonian systems in Section \ref{sec:preliminaries}, present the problem setting in Section \ref{sec:problem}, the main results in Section \ref{sec:main_results} and numerical simulations in Section  \ref{sec:numerical_sim}.  In Section \ref{sec:disc}, we discuss the implications of these conditions for future data-driven control synthesis. We conclude with Section \ref{sec:conc} with a summary and directions for future research.

\subsection{Preliminaries on Port-Hamiltonian systems}\label{sec:preliminaries}
Port-Hamiltonian (pH) systems \cite{van2014port} 
are a class of dynamical systems that model the exchange of energy between a system and its environment in a structured way. A standard continuous-time pH system is described by
\begin{equation}\label{eq:pH_general}
\begin{aligned}
    \dot{x} &= [J(x) - R(x)] \, \nabla H(x) + G(x) u\,,\\
    y &= G(x)^\top \nabla H(x)
\end{aligned}
\end{equation}
where $x \in \mathbb{R}^n$ is the state, $u \in \mathbb{R}^m$ the input, $H: \mathbb{R}^n \to \mathbb{R}$  is a smooth Hamiltonian representing the system energy. The matrix $J(x) = -J(x)^\top$ is skew-symmetric and encodes energy-preserving interconnections, while $R(x) = R(x)^\top \succeq 0$ is positive semidefinite and specifies the dissipation in the system.  The matrix $G(x)$ describes how inputs are coupled to the system. A key property of pH systems 
is passivity with storage function 
$H$. Indeed, 
\begin{equation}
\begin{aligned}
\dot{H}(x) &= \nabla H(x)^\top \dot{x}\\&= - \nabla H(x)^\top R(x) \nabla H(x) + y^\top u \\& \le y^\top u\,.
\end{aligned}
\end{equation}

\section{Problem setting}\label{sec:problem}

We consider a dataset consisting of measurements 
\be\label{eq:meas}
\mc M=\{\dot{x}_i,\, x_i,\, u_i\}_{i\in I}\,, \quad I=\{1,\dots, T\}
\ee
collected from a continuous-time time-invariant dynamical system. 
In line with our previous results for linear systems \cite{vanelli2025interpolation}, our aim is to determine whether the data are instantly consistent with a structured class of dynamical systems.

Specifically, we investigate whether, for a given matrix $G$, the data could have been generated by a noise-less port-Hamiltonian system
\begin{equation}\label{eq:pH^{(d)}}
\dot{x} = (J - R)\nabla H(x) + G u \,,
\end{equation}
where $J=-J^\top$ and $R=R\succeq 0$ are \emph{constant} matrices. We further impose $H(x)\ge 0$ for all $x \in \R^n$, which ensures that the system with output $y=G^\top \nabla H(x)$ is passive with storage function $H$.

Studying consistency with pH systems is particularly relevant in our framework, as these models naturally encode physical structure: $J$ captures energy-preserving interconnections, $R$ represents dissipation, and $H$ corresponds to stored energy. Verifying consistency therefore ensures not only that the data can be explained by a dynamical model, but also that this model satisfies intrinsic properties such as passivity, stability, and energy balance. This perspective is especially relevant when prior physical structure is available but the underlying energy function and interconnection structure are unknown, and it enables the use of passivity-based control design techniques.

In contrast to approaches relying on explicit parametrization of the Hamiltonian (e.g., basis expansions or neural approximations), we do not restrict attention to any finite-dimensional representation. Instead, we adopt a nonparametric viewpoint in which the Hamiltonian is characterized through regularity properties. Let $L > 0$ and $\mu \ge 0$. We define the class of $L$-smooth functions as
$$
\begin{aligned}
\mathcal{H}_{L}:= \{ &H \in C^1(\mathbb{R}^n) \;|\;
\\ &\|\nabla H(x) - \nabla H(y)\| \le L \|x - y\|,\ \forall x,y \in \mathbb{R}^n \}\,,
\end{aligned}
$$
and the class of $L$-smooth $\mu$-strongly convex functions as
$$
\begin{aligned}
\mathcal{H}_{\mu, L}:= \{ &H \in C^1(\mathbb{R}^n) \;|\;
\\ &H(y) \ge H(x) + \nabla H(x)^\top (y-x) + \frac{\mu}{2}\|y-x\|^2, \\& \forall x,y \in \mathbb{R}^n  \}.
\end{aligned}
$$
We further define the nonnegative subclasses
$$
\begin{aligned}
&\mathcal{H}^+_{L}:= \left\{ H \in \mathcal{H}_{L}\mid H(x) \ge 0,\ \forall x \in \mathbb{R}^n \right\},\\
&\mathcal{H}^+_{\mu,L} := \left\{ H \in \mathcal{H}_{\mu,L} \mid H(x) \ge 0,\ \forall x \in \mathbb{R}^n \right\}.
\end{aligned}
$$
Note that, for $\mu=0$, the set $\mc H_{0, L}$ (resp. $\mc H^+_{0, L}$) corresponds to $L$-\emph{smooth} convex (resp. non-negative) functions. These classes ensure well-behaved energy landscapes: smoothness bounds gradient variation, while (strong) convexity guarantees structural properties such as uniqueness of equilibria and robustness of stability.

\subsection{Objectives}\label{ss:obj}

Let us first give the following definition.

\begin{definition}[Instant data-consistency with pH system]
Let $\mc H$ be a class of \emph{non-negative} functions. The pair $(\mc M, G)$ is said to be \textit{instantly consistent} with a \emph{pH system} in the class $\mc H$ if there exist $J=-J^\top$, $R=R^\top\succeq 0$, and $H \in \mc H$ such that \eqref{eq:pH^{(d)}} holds for all triples $(\dot x_i,x_i,u_i)$ for all
$i \in I$. If $J-R$ is nonsingular, the data are said to be instantly consistent with a \emph{full-rank} pH system.
\end{definition}

Given this setup, we address two main problems.

\subsubsection{Instant data consistency with pH System}
Given $L>0$ and $\mu>0$, determine necessary and sufficient algebraic conditions for $(\mc M, G)$ to be instantly consistent with pH systems in the classes $\mc H_{L}^+$ and $\mc H_{\mu, L}^+$.

\subsubsection{Instant data consistency with full-rank pH System}
Determine necessary and sufficient conditions for the subclass where $J-R$ is nonsingular.

\subsection{Discussion of the problem assumptions}

\subsubsection{Structural assumptions} We assume prior knowledge of the input matrix $G$, which is constant and known. This corresponds to settings where the actuator configuration is known while the internal dynamics are not. For tractability, we restrict attention to pH systems with constant interconnection and damping matrices, yielding a finite-dimensional formulation while preserving the essential geometric structure. Extending the framework to the case of state-dependent input, interconnection and damping matrices is a natural direction for future work. The condition that $J-R$ is full-rank is equivalent to $\ker J \cap \ker R = \{0\}$ meaning that every direction is affected by either interconnection or dissipation (or both), preventing the operator $A=J-R$ from having degenerate null directions. A notable subclass satisfying this condition is given by systems with strictly positive definite damping ($R \succ 0$), which guarantees invertibility of $A$. In pH systems on graphs (e.g., mass–spring–damper systems) \cite{van2013port}, this corresponds to acyclic topologies.

\subsubsection{Noise and derivative estimation} In this preliminary work, we consider noise-free measurements of the state and its derivative, corresponding to a perfect interpolation setting. Handling noisy data requires suitable relaxations of the consistency conditions and is left for future work (see Remark \ref{rem:noise}).  
Noise arises in particular when $\dot{x}_i$ is not directly measured. In this case, it can be estimated using filtering techniques (e.g., low-pass or Savitzky–Golay filters) or numerical differentiation.  
Alternatively, a discrete-time formulation of the pH dynamics can be considered. Following \cite{kotyczka2021symplectic}, an implicit midpoint discretization is given by
\begin{equation}\label{eq:midpoint}
x_{i+1} =
x_i
+ \Delta_t (J-R)\nabla H\!\left(\frac{x_{i+1}+x_i}{2}\right)
+ \Delta_t G u_i,
\end{equation}
where $\Delta_t$ is the sampling time. Defining
$$
\dot{\tilde{x}}_i=\frac{x_{i+1}-x_{i}}{\Delta_t}\,, \quad \tilde x_i=\frac{x_i+x_{i+1}}{2},
$$
we obtain
$$
\dot{\tilde{x}}_i= (J-R)\nabla H(\tilde x_i)+Gu_i\,,
$$
which matches the continuous-time structure. 

\subsubsection{Instant consistency} Instant data-consistency is \emph{pointwise} and purely algebraic, treating $x_i$ and $\dot{x}_i$ as independent. It is therefore a necessary but not sufficient condition for the existence of a trajectory consistent with both the data and the pH structure, as it does not enforce $\dot{x} = \frac{dx}{dt}$. Note that, for instance, the conditions are invariant to temporal ordering: a non-pH system could be misclassified as consistent if its samples admit a permutation satisfying \eqref{eq:pH^{(d)}}.  

We remark, however, that this issue is mitigated when $\dot{x}_i$ is estimated via numerical differentiation (e.g., $\dot{x}_i \approx \frac{x_{i+1}-x_{i-1}}{2\Delta_t}$), which implicitly enforces the state-derivative relationship, at the cost of introducing noise. Alternatively, discrete-time formulations such as \eqref{eq:midpoint} provide exact temporal coupling. In this case, the conditions are exact, although the resulting model may not match the true system if the discretization assumptions are violated.  

More broadly, for sufficiently dense datasets, if $\dot{x}$ is measured and is indeed the true time derivative, we expect the discrepancy between instant and trajectory consistency to vanish, reducing the risk of false positives to highly degenerate cases. Thus, instant consistency remains a robust tool for ruling out model classes fundamentally incompatible with the observed physics. Extending the present consistency framework to explicitly enforce trajectory-level coherence in continuous time is a fundamental direction for future research.


\medskip
\subsection{Interpolation conditions}
To determine necessary and sufficient conditions on the data points to be instantly consistent with a function class $\mc H$, we rely on \textit{interpolation conditions}, a powerful tool 
developed in optimization in the context of Performance Estimation Problems \cite{taylor2017smooth, taylor2017convex, rubbens2023interpolation}. 
Consider a set of indices $I$, a class of functions $\mc H$ and a set of triples 
$$
S=\{(x_i, g_i, h_i)\}_{i \in I}, \qquad x_i\in \R^n\,, \;g_i\in \R^n\,,\; h_i\in \R\,.
$$

The set $S$ is said to be $\mc H$-\textit{interpolable} if and only if there exists a function 
$H \in \mc H$ such that 
$$
g_i = \nabla H(x_i) \,, \quad H(x_i) = h_i, \qquad \forall i \in I.
$$
The following interpolation conditions, adapted from \cite{taylor2017smooth}, characterize when such a function exists. 
\begin{proposition}[Smooth (convex) interpolation] \label{pr:sm_int}
The set $S$ is 
\begin{itemize}
\item $\mathcal{H}_{L}$-interpolable if and only if
    \begin{equation}\label{eq:ic_smooth}
    \begin{aligned}
    h_i - h_j \ge&  -\frac{L}{4} \|x_i - x_j\|^2 +\frac{1}{2} (g_i + g_j)^\top (x_i - x_j) \\&+ \frac{1}{4L} \|g_i - g_j\|^2, \qquad \forall i,j \in I.
    \end{aligned}
    \end{equation}
    \item  $\mc H_{\mu,L}$-interpolable if and only if
\begin{equation}\label{eq:ic_mu_convex}
\begin{aligned}
h_i-h_j
\ge\;
&g_j^\top(x_i-x_j) \\
&+
\frac{1}{2(1-\mu/L)}
\Big(
\frac{1}{L}\|g_i-g_j\|^2
+\mu\|x_i-x_j\|^2 \\
&\qquad
-\frac{\mu}{L}(g_j-g_i)^\top(x_j-x_i)
\Big),
\quad \forall i,j\in I \,.
\end{aligned}
\end{equation}
\end{itemize}
For $\mu=0$, we find that $S$ is $\mathcal{H}_{0,L}$-interpolable if and only if
    \begin{equation}\label{eq:ic_convex}
    h_i- h_j  \ge g_j^\top (x_i - x_j) + \frac{1}{2L} \|g_i - g_j\|^2, \quad \forall i,j \in I.
    \end{equation}\end{proposition}\medskip
These conditions provide \textit{necessary and sufficient conditions} for the observed data to be instantly consistent with a Hamiltonian $H$ with certain smoothness and convexity properties. 

\section{Main results}\label{sec:main_results}
In this section, we derive and present our main results addressing the objectives outlined in Section \ref{ss:obj}.
Let us first introduce some notation that will prove convenient in the sequel. Given the measurements $\mc M=\{\dot x_i, x_i, u_i\}_{i\in I}$ and the input matrix $G$, define
\begin{equation}\label{eq:y}
Y:=\bma y_1& \dots& y_T\ema, 
\qquad 
y_i := \dot x_{i}-Gu_i \,. 
\end{equation}
Then, the pair $(\mc M, G)$ is instantly consistent with the pH dynamics in \eqref{eq:pH^{(d)}} if and only if
\begin{equation}\label{eq:y_dyn}
y_i = (J-R)\nabla H(x_i), \qquad \forall i\in I.
\end{equation}
If we further denote $\Gamma=[\nabla H(x_1)\,,\dots\,, \nabla H(x_T)]$, we can express \eqref{eq:y_dyn} in the compact form $$Y=(J-R)\Gamma\,.$$
As is standard in port-Hamiltonian systems, the next lemma shows that one can equivalently work with a single matrix having negative semidefinite symmetric part.
\begin{lemma}\label{lemma:eq_lc}
Let $Y\in \R^{n\times T}$ and $\Gamma\in \R^{n\times T}$. Then, there exist matrices $J=-J^\top$ and $R=R^\top\succeq0$ such that
$Y=(J-R)\Gamma$
if and only if there exists a matrix $A$ satisfying
\[
Y=A\Gamma, \qquad A+A^\top\preceq0\,.
\]
\end{lemma}
\medskip
\begin{proof}
$\Rightarrow)$  
If $Y=(J-R)\Gamma$, define $A:=J-R$. Then
$A+A^\top=(J-R)+(J-R)^\top=-2R\preceq0$,
so condition (2) holds.\\
$\Leftarrow)$ 
Assume $Y=A\Gamma$ with $A+A^\top\preceq0$. Define
\[
J=\frac{A-A^\top}{2},\qquad
R=-\frac{A+A^\top}{2}.
\]
Then $J=-J^\top$ and $R=R^\top\succeq0$, and clearly $A=J-R$. Hence $Y=(J-R)\Gamma$.
\end{proof}
\medskip
We now state our general consistency result, giving necessary and sufficient algebraic conditions for the measurements and the input matrix to be instantly consistent with a pH system with an $L$-smooth Hamiltonian and $\mu$-strongly convex $L$-smooth Hamiltonian, respectively.
\medskip
\begin{theorem}[Data-consistency with pH system]
\label{th:feasibility}
Let $\mathcal{M}$ be the set of measurements and $G$ the input matrix, with $Y$ defined as in \eqref{eq:y}. 
The pair $(\mathcal{M}, G)$ is instantly consistent with a pH system in the class $\mathcal{H}_{L}$ if and only if there exists a triple
$$
(A, \{g_i,h_i\}_{i\in I})
$$
with $A \in \mathbb{R}^{n \times n}$, $g_i \in \mathbb{R}^n$, and $h_i \in \mathbb{R}$ such that
\begin{equation}\label{eq:dc_smooth}
\begin{cases} 
Y = A [g_1 \ \dots \ g_T], \\
A + A^\top \preceq 0, \\
(x_i, g_i, h_i)_{i \in I} \text{ satisfy \eqref{eq:ic_smooth}.
}
\end{cases}
\end{equation}
Moreover, any triple $(A, \{g_i, h_i\})$ satisfying \eqref{eq:dc_smooth} admits at least one port-Hamiltonian realization, that is, there exists a pH system instantly consistent with $(\mathcal{M}, G)$ with $J = -J^\top$, $R = R^\top \succeq 0$, and $H \in \mathcal{H}_{L}$,  
satisfying
\be \label{eq:cons}
A = J - R, \quad \nabla H(x_i) = g_i, \quad H(x_i) = h_i+c \quad \forall i \in I\,,
\ee
for some $c$ in $\R$.
The same result holds for the class $\mathcal{H}_{\mu,L}$, $\mu \ge 0$, with \eqref{eq:ic_mu_convex} replacing \eqref{eq:ic_smooth}.
\end{theorem}
\medskip
\begin{proof}
We will prove the statement only for $\mc H=\mc H_{L}$ as the proof for $\mc H_{\mu, L}$ is analogous. By definition, the pair $(\mc M, G)$ is instantly consistent with a pH system in the class $\mc H_{L}$ if and only if there exist a Hamiltonian $H$ in $\mc H_{L}$ and matrices $J=-J^\top$, $R\succeq 0$ such that \eqref{eq:pH^{(d)}}
is satisfied for all $\{\dot x_i, x_i, u_i\}_{i\in I}$, or equivalently, such that \eqref{eq:y_dyn} is satisfied for all $\{x_i, y_i\}_{i\in I}$.   
Let
$$g_i=\nabla H(x_i),\qquad
\Gamma=[g_1,\dots,g_T]\,.$$
Then $Y=(J-R)\Gamma$ with $J=-J^\top$, $R\succeq 0$.  By Lemma \ref{lemma:eq_lc}, this is equivalent to the condition $Y=A\Gamma$ with $A+A^\top\preceq0$. Finally, by Proposition \ref{pr:sm_int}, there exists $H$ in $\mc H_{L}$ such that $\nabla H(x_i)=g_i$ for all $i$ if and only if there exists $\{h_i\}_I$ such that \eqref{eq:ic_smooth} are satisfied for  $\{x_i, g_i, h_i\}_I$. Combining these conditions gives the necessary and sufficient conditions in \eqref{eq:dc_smooth}. Note that if such $H$ is not nonnegative, there exists a $c>0$ such that $h^{(d)}(x)=H(x)+c$ satisfies the feasibility problem \eqref{eq:dc_smooth} and is nonnegative.
\end{proof}
\begin{remark}[Interpretation of the decision variables]
The variables in Theorem \ref{th:feasibility} admit a natural interpretation. The vectors $g_i$ represent candidate gradients $\nabla H(x_i)$ and $h_i$ the corresponding energy values, while the matrix $A = J - R$ encodes the interconnection and dissipation structure, with $A + A^\top \preceq 0$. From this perspective, the feasibility problem~\eqref{eq:dc_smooth} enforces three coupled requirements: (i) the existence of a linear operator $A$ satisfying the interconnection and dissipation structure of port-Hamiltonian systems, (ii) the existence of gradients $g_i$ that are consistent with a smooth (convex) Hamiltonian through interpolation conditions, and (iii) instant consistency of the data with the model through the relation $y_i = A g_i$ for all $i \in I$.
\end{remark}

For $\mu=0$, we find necessary and sufficient conditions for the measurements and the input matrix to be instantly consistent with a pH system with a convex $L$-smooth Hamiltonian $H$. 
\begin{corollary}[Data-consistency with convex pH system]
    For $L>0$, the pair $(\mc M, G)$ is instantly consistent with the pH system in the class $\mc H_{0,L}$ of convex $L$-smooth Hamiltonians if and only if there exists a triple $
    (A, \{g_i,h_i\}_{i\in I})$
    with $A \in \mathbb{R}^{n \times n}$, $g_i \in \mathbb{R}^n$, $h_i \in \mathbb{R}$ such that
\be\label{eq:dc_convex2}
\begin{cases} 
Y= A [ g_1\,, \dots\,, g_T] \\
A+A^\top \preceq 0\\
h_i-h_j
\ge\;
g_j^\top(x_i-x_j)+
\frac{1}{2L}\|g_i-g_j\|^2,
\quad \forall  i\neq j \,.
\end{cases}
\ee
\end{corollary}

The feasibility problems \eqref{eq:dc_smooth} and the corresponding one for $\mc H_{\mu, L}$ provide a complete characterization of all pH models instantly consistent with  $(\mc M, G )$ and the class $\mc H$. However, they are nonconvex quadratically constrained quadratic programs (QCQP), as they are bilinear in $A$ and $\Gamma$ due to the constraint $Y = A\Gamma$,
which couples the unknown matrix $A$ and the gradients $\Gamma=[g_1,\dots,g_T]$. 

\begin{remark}[Parametrization of $A$]
In many practical applications, $J$ and $R$ are known in parametric form. Note that, if parametric sets $\mc J_\theta$ and $\mc R_\theta$ are known, this can be easily added in the feasibility program as a further constraint. In particular, if 
$\mc J_\theta$ and  $\mc R_\theta$ are such that all matrices in $\mc J_\theta$ are skew-symmetric and $\mc R_\theta$ are positive semidefinite, then the feasibility program in \eqref{eq:dc_smooth} reduces to
\be
\begin{cases} 
Y= A_\theta[g_{1}, \dotsm g_{T}]\,, \\
A_\theta=J_\theta-R_\theta, \quad J_\theta\in \mc J_\theta\,,\, R_\theta\in \mc R_{\theta}\\
 \{x_i, g_i, h_i\} \text{ satisfy }  \eqref{eq:ic_smooth} \textit{ (resp. \eqref{eq:ic_mu_convex})}
\end{cases}
\ee  
If the parametrization is linear in $\theta$, this becomes a convex QCQP regardless of the rank of $J-R$.
\end{remark}

On the other hand, the subproblem of whether there exists a \emph{full-rank} pH system instantly consistent with the data admits a convex SDP reformulation,  as proved in following result. The key idea is to eliminate the bilinearity by expressing gradients explicitly via the inverse operator


\begin{theorem}[SDP formulation for full-rank pH system]\label{th:feasibility2}
Let $\mathcal{M}$ be the set of measurements and $G$ the input matrix, with $Y$ defined as in \eqref{eq:y} and $$
\Delta x_{ij} := x_i - x_j\,,\;
\Delta y_{ij} := y_i - y_j\,.
$$ 
The pair $(\mathcal{M}, G)$ is instantly consistent with a \emph{full-rank} pH system in the class $\mathcal{H}_{L}$
 if and only if there exists a pair $$(B, \{ h_i\}_{i\in I})$$ with $B$ in $\R^{n\times n}$ and $h_i$ in $\R$ satisfying
\begin{equation}\label{eq:dc_smooth_sdp}
\begin{cases}
B + B^\top \prec 0\,, \\
\begin{bmatrix}
\phi_{ij}
& (B \Delta y_{ij})^\top \\
B \Delta y_{ij}& I
\end{bmatrix}
\succeq 0\,,
\end{cases}
\end{equation}
where 
\be\label{eq:phiL}
\phi_{ij}:=4L\big(h_i-h_j
- \frac1{2} (B(y_i+y_j))^\top \Delta x_{ij}
+ \tfrac{L}{4}\|\Delta x_{ij}\|^2\big)\,.
\ee
The same result holds for the class $\mathcal{H}_{\mu,L}$, $\mu \ge 0$, with
\be\label{eq:phimuL}
\begin{aligned}
\phi_{ij}:=&2(L-\mu)(h_i-h_j-(By_j)^\top \Delta x_{ij})-L\mu \|\Delta x_{ij}\|^2\\&+\mu (B \Delta y_{ij})^\top \Delta x_{ij}   
\end{aligned}
\ee
Moreover, any triple $(A, \{g_i, h_i\}_{i\in I})=(B^{-1}, \{By_i, h_i\}_{i\in I})$ with $(B, \{h_i\}_{i\in I})$ satisfying \eqref{eq:dc_smooth_sdp} admits at least one port-Hamiltonian realization satisfying \eqref{eq:cons}. 
\end{theorem}
\medskip
\begin{proof}
We will prove the statement only for $\mc H=\mc H_{L}$ as the proof for $\mc H_{\mu, L}$ is analogous. By Theorem \ref{th:feasibility}, the pair $(\mathcal{M}, G)$ is instantly consistent with a port-Hamiltonian system in $\mathcal{H}_{L}$ if and only if the feasibility problem in \eqref{eq:dc_smooth} admits at least one solution. To impose that this solution is full rank, we further impose $A+A^\top \prec 0$. Since the symmetric part of $A$ is strictly negative definite, $A$ is nonsingular. We can therefore define its inverse $B = A^{-1}$. Multiplying the consistency equation $Y = A\Gamma$ by $B$ from the left yields the explicit expression for the gradients:
\begin{equation}
\Gamma = BY \implies g_i = By_i, \quad \forall i \in I.
\end{equation}
Substituting $g_i = By_i$ and $g_j = By_j$ into the $L$-smooth interpolation condition \eqref{eq:ic_smooth}, gives
 that the condition simplifies to:
\begin{equation}
\phi_{ij} \ge (B\Delta y_{ij})^\top I (B\Delta y_{ij}).
\end{equation}
with $\phi_{ij}$ as in \eqref{eq:phiL}.
The inequality above is a quadratic constraint in the variable $B$. Since $I \succ 0$, we apply the Schur complement to obtain the equivalent convex reformulation
\begin{equation}
\begin{bmatrix} \phi_{ij} & (B\Delta y_{ij})^\top \\ B\Delta y_{ij} & I \end{bmatrix} \succeq 0, \quad \forall i \neq j.
\end{equation}
Finally, since $A+A^\top   \prec 0$ and $B=A^{-1}$, we find the equivalent condition $B+B^\top  \prec 0$. This completes the proof.
\end{proof}
\begin{algorithm}
\caption{Full-rank \textit{(convex)} pH instant data-consistency}
\label{alg:sdp}
\begin{algorithmic}[1]
\Require Measurements $\mathcal{M}=\{x_i,\dot x_i,u_i\}_{i\in I}$, input matrix $G$, smoothness parameter $L$, (\textit{convexity parameter} $\mu$)
\Ensure Feasible / Infeasible, (optionally) $(A,\{g_i,h_i\}_{i\in I})$

\State Compute $y_i = \dot x_i - G u_i$, $\forall i \in I$
\State Form differences $\Delta x_{ij} = x_i - x_j$, $\Delta y_{ij} = y_i - y_j$

\State Solve the SDP:
\[
\text{find } B \in \mathbb{R}^{n\times n}, \; \{h_i\}_{i\in I}, h_i \in \mathbb{R}
\]
such that
\[
B + B^\top \prec 0
\]
and for all $i \neq j$
\[
\begin{bmatrix}
\phi_{ij} & (B \Delta y_{ij})^\top \\
B \Delta y_{ij} & I
\end{bmatrix}
\succeq 0
\]
where
\[
\phi_{ij} = 4L\Big(h_i - h_j 
- \tfrac{1}{2}(B(y_i+y_j))^\top \Delta x_{ij}
+ \tfrac{L}{4}\|\Delta x_{ij}\|^2\Big)
\]
\textit{($\phi_{ij}$ as in \eqref{eq:phimuL} if $\mu$ is provided)}
\medskip

\If{SDP is feasible}
    \State Recover $g_i = B y_i$
    \State Set $A = B^{-1}$
    \State \Return \textbf{Feasible}, $(A,\{g_i,h_i\}_{i\in I})$
\Else
    \State \Return \textbf{Infeasible}
\EndIf
\end{algorithmic}
\end{algorithm}
Theorem \ref{th:feasibility2} provides necessary and sufficient algebraic conditions for the measurements and the input matrix to be instantly consistent with a \emph{full-rank} pH system with an $L$-smooth Hamiltonian and $\mu$-strongly convex $L$-smooth Hamiltonian, respectively.
These feasibility problems can be translated into SDP programs as shown in Algorithm \ref{alg:sdp}. 

\begin{remark}[Noise relaxation via slack variables]\label{rem:noise}
The perfect interpolation setting can be relaxed to handle noisy measurements by introducing slack variables into the constraints. Specifically, one can introduce a variable $V \in \mathbb{R}^{n \times T}$ (with columns $v_i \in \mathbb{R}^n$) to relax the data consistency constraint to $g_i = By_i + v_i$, which accounts for noise in the state derivatives or input measurements. Additionally, a second slack variable $E \in \mathbb{R}^{T \times T}$ 
can be added to $\Phi_{ij}$ in \eqref{eq:phiL} (resp. \eqref{eq:phimuL}) to relax the interpolation inequalities  
and the requirements on smoothness or convexity as a consequence. This allows for a robust consistency check where one minimizes a weighted sum of the residuals to find the best-fit port-Hamiltonian structure.
\end{remark}

\begin{remark}[Computational complexity]
The feasibility problems presented in Theorem 2 and Algorithm 1 are formulated as convex semidefinite programs (SDPs). 
The computational complexity is dominated by the $O(T^2)$ linear matrix inequalities (LMIs) required to enforce the interpolation conditions across all pairs of data points. For datasets of moderate size (e.g., $T \approx 200$), the problem remains numerically tractable on standard hardware.
The addition of the slack variables mentioned in Remark \ref{rem:noise} preserves the convexity of the problem and does not impact the computational complexity. 
\end{remark}
\section{Numerical simulations}\label{sec:numerical_sim}
In this section, we evaluate Algorithm~1 (full-rank (convex) pH instant data-consistency) 
on two representative nonlinear dynamical systems: a mass-spring-damper system admitting a pH representation with \emph{convex} Hamiltonian and a nonlinear electrostatic microactuator with \emph{non-convex} Hamiltonian. The objective is to assess (i) the ability of the framework to correctly identify port-Hamiltonian structure from data, (ii) the sensitivity of the feasibility conditions to smoothness and convexity parameters, and (iii) the impact of derivative estimation, instant consistency and sampling.
\begin{figure}
\centering
\includegraphics[width=0.5\textwidth]{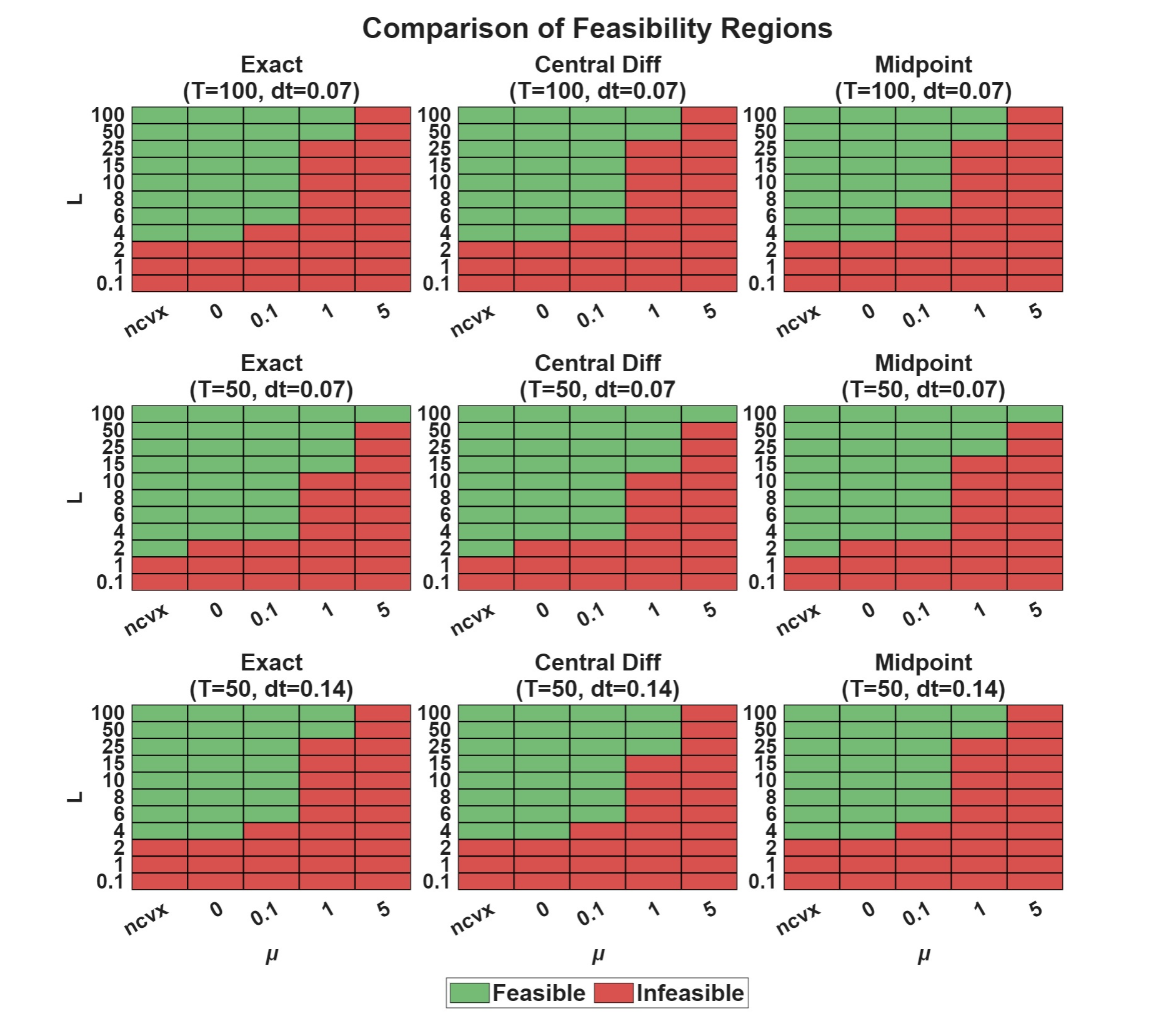}
\caption{Feasibility regions of Algorithm 1 for the convex mass–spring–damper system (Example \ref{ex:1}), showing correct identification of the underlying pH structure across different smoothness and convexity parameters.}
\label{fig:1}
\end{figure}
\begin{example}[Mass–Spring–Damper System]\label{ex:1}
We consider a nonlinear mass–spring–damper system with state $x = [q\; p]^\top$, where $q$ is the position and $p = m\dot q$ is the momentum. The dynamics is given by
\[
\dot q = \frac{p}{m}, \qquad
\dot p = -k q - \alpha q^3 - \frac{b}{m}p + u\,,
\]
with $m=1$, $k=1$, $\alpha=1$ and $b=0.5$.
The system admits a pH representation with a \emph{strongly convex Hamiltonian}
\begin{equation}
H(q,p) = \frac{p^2}{2m} + \frac{1}{2}k q^2 + \frac{1}{4}\alpha q^4,
\end{equation}
and structure matrices
$$
J = \bma 0 & 1 \\ -1 & 0 \ema,\quad
R = \bma 0 & 0 \\ 0 & b \ema ,\quad
G = \bma  0 \\ 1 \ema\,.
$$ This example serves as a benchmark where the pH structure is well-conditioned.
\end{example}

\begin{example}[Nonlinear Electrostatic Microactuator]\label{ex:2}
Following \cite{BeckersCDC2023}, we consider an actuator consisting of a movable plate with mass $m=1$, a nonlinear spring with position-dependent stiffness $k(x_1)$, a capacitive electrostatic actuator with input voltage $u$ and a linear damping $b=0.5$.
We define the state vector
\[
x = 
\begin{bmatrix}
x_1 \\ x_2 \\ x_3
\end{bmatrix}
=
\begin{bmatrix}
\text{air gap} \\ \text{momentum} \\ \text{charge}
\end{bmatrix}.
\]
The dynamics can be expressed as a pH system as in \eqref{eq:pH^{(d)}} 
with a \emph{non-convex} Hamiltonian: 
\begin{equation}
H(x) = \frac{1}{4} (x_1 - x^*_1)^4 + \frac{1}{2} m x_2^2 + \frac{x_3^2 x_1}{2 A \varepsilon}\,,
\end{equation}
where $x_1^* = 1$ is the nominal equilibrium air gap, $A = 1$ is the plate area, and $\varepsilon = 1$ is the permittivity in the gap (we refer to \cite{maithripala2003nonlinear,beckers2023data} for further details). The interconnection and damping matrices are:
\[
J = 
\begin{bmatrix}
0 & 1 & 0 \\
-1 & 0 & 0 \\
0 & 0 & 0
\end{bmatrix}, \quad
R = 
\begin{bmatrix}
0 & 0 & 0 \\
0 & b & 0 \\
0 & 0 & \frac{1}{r}
\end{bmatrix}, \quad
G = 
\begin{bmatrix}
0 \\ 0 \\ \frac{1}{r}
\end{bmatrix},
\]
where $r = 1$ is the input resistance.

\end{example}
\begin{figure}
\centering
\includegraphics[width=0.5\textwidth]{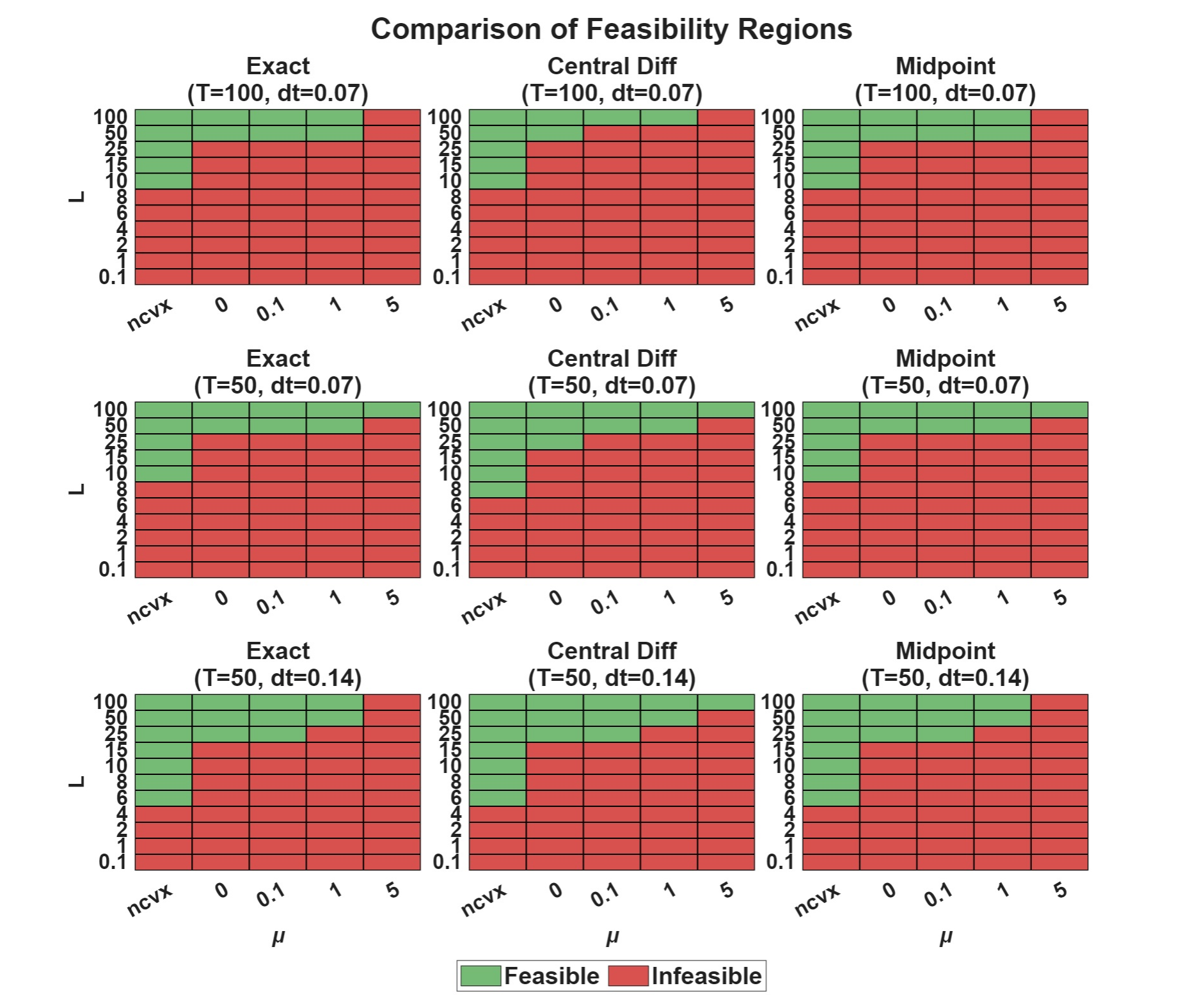}
\caption{Feasibility regions of Algorithm 1 for the non-convex electrostatic microactuator (Example \ref{ex:2}), illustrating rejection of convex pH structure for small smoothness bounds and the emergence of false feasibility under overly relaxed smoothness constraints.}
\label{fig:2}
\end{figure}

First, we collect a set of training data for both the systems in Example \ref{ex:1} and \ref{ex:2}. For this purpose, we use a sinusoidal input signal, i.e., $u(t) = \sin(t)$. The systems are initialized with $x(0) = [0,1]^\top$ and $x(0) = [0,0,1]^\top$, respectively, and 100 data pairs $\{t_i, x(t_i)\}$ are recorded between $0\,\mathrm{ms}$ and $t_\text{span}\,\mathrm{ms}$ 
with constant time spacing $dt=\frac{t_\text{span}}{T}$.  
Specifically, to investigate the effect of number of data-points and sampling density, we test $T=100$ with $dt=0.06$, $T=50$ with $dt=0.06$ and $T=100$ with $dt=0.14$. Furthermore, we evaluate three derivative estimation methods: \emph{exact derivatives} (ground-truth), \emph{central finite differences} ($O(\Delta t^2)$), and the structure-preserving \emph{midpoint} discretization \eqref{eq:midpoint}. 
We then test instant data-consistency with a full-rank pH system (Algorithm \ref{alg:sdp}) 
over a range of smoothness and convexity parameters:
$$
\begin{aligned}
L &\in \{0.1, 1, 2, 4, 6, 8, 10, 15, 25, 50, 100\}, \\
\mu &\in \{-1, 0, 0.1, 1, 5\}.
\end{aligned}
$$
To ensure numerical stability and prevent ill-conditioned solutions where $\|A\|$ is excessively large, we regularize the SDP by imposing $\|A\| \leq \kappa$, $\kappa=5$, via the constraint:$$-(B + B^\top) - \frac{2}{\kappa} I \succeq 0.$$
The results are shown in Fig. \ref{fig:1} and Fig. \ref{fig:2}. 

Overall, the experimental results confirm that the proposed data-consistency conditions correctly identify the underlying system dynamics for the considered class of port-Hamiltonian systems. Notably, when testing the convex pH system in Example \ref{ex:1}, we find feasibility with $\mu$-strongly convex Hamiltonians even at relatively tight smoothness bounds (e.g., $L=4$ when $\mu=0$, and $L=6$ when $\mu=0.1$). 
In contrast, for the non-convex system in Example~4.2, the method correctly rejects $\mu$-strong convexity for small values of $L$. However, when the smoothness constraint is relaxed, the feasible set becomes over-approximated, eventually leading to incorrect convex feasibility at large values ($L=50$). 

Regarding derivative estimation, all methods show similar overall behavior. The midpoint discretization most closely matches the true feasibility boundary, while central finite differences are slightly less accurate and may introduce false feasibility at larger $L$ due to discretization error.

Finally, regarding data distribution, the total number of sampled points has a negligible impact on the feasibility boundaries of Example \ref{ex:2} (non-convex), while it has a relevant impact on Example \ref{ex:1}. This is likely due to the fact that the initial data points were easily interpolated using a convex gradient structure. The sampling interval $dt$ also proves to be influential: coarser sampling densities (i.e., larger values of $dt$) inherently reduce accuracy, causing the aforementioned false convexity to emerge at lower smoothness bounds ($L = 25$ instead of $L = 50$).  

\section{Discussion: Towards Data-Driven Control via Interpolation}\label{sec:disc}
While the focus of this paper
is on instant consistency analysis rather than control synthesis, the interpolation-based viewpoint naturally suggests a pathway toward data-driven control design. To stabilize port-Hamiltonian systems, a widely used approach is \emph{Interconnection and Damping Assignment Passivity-Based Control} (IDA-PBC) \cite{ortega2002interconnection}. The main idea is to reshape both the energy function and the system structure such that the closed-loop dynamics take a desired pH form with a minimum at the target equilibrium. In the following, we outline how the proposed framework could be extended to incorporate energy-shaping objectives as in IDA-PBC design.

At a high level, the interpolation conditions ensure that the measured data $\{x_i, \dot{x}_i, u_i\}$ are consistent with dynamics of the form \eqref{eq:pH^{(d)}} with $H$ in a structured function class. In Section \ref{sec:main_results}, we showed that this can be formulated as a feasibility problem over gradients $g_i = \nabla H(x_i)$ that satisfy both the data and the regularity properties of $H$. This viewpoint extends naturally to control by asking whether the data are compatible with a \emph{desired} closed-loop structure of the form
\[
\dot{x} = (J^{(d)} - R^{(d)}) \nabla H^{(d)}(x),
\]
where $H^{(d)}$ has a minimum at a prescribed equilibrium $x^\ast$ and satisfies suitable regularity properties (e.g., local strong convexity). Following IDA-PBC principles, this leads to constraints of the form
\[
G^\perp Y = G^\perp A^{(d)} g^{(d)}_i, \quad g^{(d)}_i = \nabla H^{(d)}(x_i),
\]
with $Y$ as in \eqref{eq:y}, enforcing that the components of the dynamics not influenced by the input match the desired energy-shaped structure. In this way, the control objective is embedded directly into interpolation conditions on gradients and energy values.

This perspective offers several conceptual advantages. First, it bypasses explicit model identification or parametrization, replacing it with feasibility conditions that can be checked directly from data. Second, it aligns naturally with passivity-based control, where stability is encoded through energy shaping and dissipation, e.g.,
\[
A^{(d)} + (A^{(d)})^\top \preceq 0, \quad \nabla H^{(d)}(x^\ast) = 0.
\]
Third, it suggests the possibility of constructing controllers by interpolating gradients and energy values consistent with both the data and the desired equilibrium behavior.

At the same time, significant challenges remain. The resulting constraints are more complex than in the open-loop case, involving bilinear terms such as $A^{(d)} g^{(d)}_i$ and additional degrees of freedom induced by the input directions, 
Moreover, interpolation conditions enforce consistency only at sampled points and do not guarantee desirable behavior between samples. 
Another challenge is that consistency with IDA-PBC design is only a necessary condition, 
as it certifies compatibility with an IDA-PBC design at the data level, but does not yet guarantee stabilizability. Bridging this gap will likely require further assumption on the data. 

Despite these challenges, the interpolation-based viewpoint provides a promising direction for data-driven control. It suggests a unified framework in which system structure, energy shaping, and control objectives are encoded directly as constraints on data. Further developments may enable the synthesis of passivity-based controllers directly from measurements, particularly when first-principles models are unavailable.

\section{Conclusions}\label{sec:conc}
In this paper, we developed an interpolation-based framework to assess instant data consistency with port-Hamiltonian systems directly from measured trajectories, without requiring explicit model identification. By leveraging interpolation conditions for smooth and convex function classes, we derived necessary and sufficient algebraic conditions for consistency with structured dynamical models. We further provided a convex semidefinite programming formulation for the full-rank case, enabling efficient verification from data. 

Beyond consistency analysis, we outlined how this framework may be extended toward data-driven control design, particularly within passivity-based control approaches. An important limitation of the current framework is its pointwise nature, which does not enforce trajectory-level coherence. Extending these results to noisy and trajectory-consistent settings remains an important direction for future work.




	\bibliographystyle{ieeetr}
	\bibliography{bib}

@article{berberich2020data,
  title={Data-driven model predictive control with stability and robustness guarantees},
  author={Berberich, Julian and K{\"o}hler, Johannes and M{\"u}ller, Matthias A and Allg{\"o}wer, Frank},
  journal={IEEE Transactions on Automatic Control},
  volume={66},
  number={4},
  pages={1702--1717},
  year={2020},
  publisher={IEEE}
}

@article{bousselmi2024interpolation,
  title={Interpolation conditions for linear operators and applications to performance estimation problems},
  author={Bousselmi, Nizar and Hendrickx, Julien M and Glineur, Fran{\c{c}}ois},
  journal={SIAM Journal on Optimization},
  volume={34},
  number={3},
  pages={3033--3063},
  year={2024},
  publisher={SIAM}
}

@article{niknejad2023physics,
  title={Physics-informed data-driven safe and optimal control design},
  author={Niknejad, Nariman and Modares, Hamidreza},
  journal={IEEE Control Systems Letters},
  volume={8},
  pages={285--290},
  year={2023},
  publisher={IEEE}
}

@article{campi2002vrft,
  author  = {Campi, M. C. and Lecchini, A. and Savaresi, S. M.},
  title   = {Virtual reference feedback tuning: a direct method for the design of feedback controllers},
  journal = {Automatica},
  volume  = {38},
  number  = {8},
  pages   = {1337--1346},
  year    = {2002}
}

@article{tanaskovic2017data,
  author  = {Tanaskovic, M. and Fagiano, L. and Novara, C. and Morari, M.},
  title   = {Data-driven control of nonlinear systems: An on-line direct approach},
  journal = {Automatica},
  volume  = {75},
  pages   = {1--10},
  year    = {2017}
}

@article{fliess2013modelfree,
  author  = {Fliess, M. and Join, C.},
  title   = {Model-free control},
  journal = {International Journal of Control},
  volume  = {86},
  number  = {12},
  pages   = {2228--2252},
  year    = {2013}
}

@article{fraile2020stabilization,
  author  = {Fraile, L. and Marchi, M. and Tabuada, P.},
  title   = {Data-driven stabilization of SISO feedback linearizable systems},
  journal = {arXiv preprint arXiv:2003.14240},
  year    = {2020}
}

@article{berberich2024overview,
  author  = {Berberich, J. and Allgöwer, F.},
  title   = {An overview of systems-theoretic guarantees in data-driven MPC},
  journal = {Annual Review of Control, Robotics, and Autonomous Systems},
  volume  = {8},
  year    = {2024}
}

@article{proctor2016dmdc,
  author  = {Proctor, J. L. and Brunton, S. L. and Kutz, J. N.},
  title   = {Dynamic mode decomposition with control},
  journal = {SIAM Journal on Applied Dynamical Systems},
  volume  = {15},
  number  = {1},
  pages   = {142--161},
  year    = {2016}
}

@article{korda2018koopman,
  author  = {Korda, M. and Mezić, I.},
  title   = {Linear predictors for nonlinear dynamical systems: Koopman operator meets MPC},
  journal = {Automatica},
  volume  = {93},
  pages   = {149--160},
  year    = {2018}
}

@article{bisoffi2020bilinear,
  author  = {Bisoffi, A. and De Persis, C. and Tesi, P.},
  title   = {Data-based stabilization of unknown bilinear systems with guaranteed basin of attraction},
  journal = {Systems \& Control Letters},
  volume  = {145},
  pages   = {104788},
  year    = {2020}
}

@article{guo2021polynomial,
  author  = {Guo, M. and De Persis, C. and Tesi, P.},
  title   = {Data-driven stabilization of nonlinear polynomial systems with noisy data},
  journal = {IEEE Transactions on Automatic Control},
  volume  = {67},
  number  = {8},
  pages   = {4210--4217},
  year    = {2021}
}

@inproceedings{strasser2021beyond,
  author    = {Strässer, R. and Berberich, J. and Allgöwer, F.},
  title     = {Data-driven control of nonlinear systems: Beyond polynomial dynamics},
  booktitle = {2021 IEEE CDC},
  pages     = {4344--4351},
  year      = {2021}
}

@inproceedings{alsalti2021flat,
  author    = {Alsalti, M. and Berberich, J. and López, V. G. and Allgöwer, F. and Müller, M. A.},
  title     = {Data-based system analysis and control of flat nonlinear systems},
  booktitle = {2021 IEEE CDC},
  pages     = {1484--1489},
  year      = {2021}
}

@article{monshizadeh2025versatile,
  title={A versatile framework for data-driven control of nonlinear systems},
  author={Monshizadeh, Nima and De Persis, Claudio and Tesi, Pietro},
  journal={IEEE Transactions on Automatic Control},
  year={2025},
  publisher={IEEE}
}

@article{van2014port,
  title={Port-Hamiltonian systems theory: An introductory overview},
  author={Van Der Schaft, Arjan and Jeltsema, Dimitri},
  journal={Foundations and Trends{\textregistered} in Systems and Control},
  volume={1},
  number={2-3},
  pages={173--378},
  year={2014},
  publisher={Emerald Publishing Limited Boston—Delft}
}

@article{ortega2002interconnection,
  title={Interconnection and damping assignment passivity-based control of port-controlled Hamiltonian systems},
  author={Ortega, Romeo and Van Der Schaft, Arjan and Maschke, Bernhard and Escobar, Gerardo},
  journal={Automatica},
  volume={38},
  number={4},
  pages={585--596},
  year={2002},
  publisher={Elsevier}
}

@article{kotyczka2021symplectic,
  title={Symplectic discrete-time energy-based control for nonlinear mechanical systems},
  author={Kotyczka, Paul and Thoma, Tobias},
  journal={Automatica},
  volume={133},
  pages={109842},
  year={2021},
  publisher={Elsevier}
}

@inproceedings{beckers2023data,
  title={Data-driven Bayesian control of port-Hamiltonian systems},
  author={Beckers, Thomas},
  booktitle={2023 62nd IEEE Conference on Decision and Control (CDC)},
  pages={8708--8713},
  year={2023},
  organization={IEEE}
}

@inproceedings{rubbens2023interpolation,
  title={Interpolation constraints for computing worst-case bounds in performance estimation problems},
  author={Rubbens, Anne and Bousselmi, Nizar and Colla, S{\'e}bastien and Hendrickx, Julien M},
  booktitle={2023 62nd IEEE Conference on Decision and Control (CDC)},
  pages={3015--3022},
  year={2023},
  organization={IEEE}
}

@book{van2000l2,
  title={L2-gain and passivity techniques in nonlinear control},
  author={Van der Schaft, Arjan},
  year={2000},
  publisher={Springer}
}

@phdthesis{taylor2017convex,
  title={Convex interpolation and performance estimation of first-order methods for convex optimization},
  author={Taylor, Adrien B},
  year={2017},
  school={UCLouvain}
}

@article{vanelli2025interpolation,
  title={Interpolation Conditions for Data Consistency and Prediction in Noisy Linear Systems},
  author={Vanelli, Martina and Monshizadeh, Nima and Hendrickx, Julien M},
  journal={arXiv preprint arXiv:2504.08484},
  year={2025}}

@article{bisoffi2022data,
  title={Data-driven control via Petersen’s lemma},
  author={Bisoffi, Andrea and De Persis, Claudio and Tesi, Pietro},
  journal={Automatica},
  volume={145},
  pages={110537},
  year={2022},
  publisher={Elsevier}
}

@article{taylor2017smooth,
  title={Smooth strongly convex interpolation and exact worst-case performance of first-order methods},
  author={Taylor, Adrien B and Hendrickx, Julien M and Glineur, Fran{\c{c}}ois},
  journal={Mathematical Programming},
  volume={161},
  number={1},
  pages={307--345},
  year={2017},
  publisher={Springer}
}

@article{markovsky2021behavioral,
  title={Behavioral systems theory in data-driven analysis, signal processing, and control},
  author={Markovsky, Ivan and D{\"o}rfler, Florian},
  journal={Annual Reviews in Control},
  volume={52},
  pages={42--64},
  year={2021},
  publisher={Elsevier}
}

@article{van2020willems,
  title={Willems’ fundamental lemma for state-space systems and its extension to multiple datasets},
  author={Van Waarde, Henk J and De Persis, Claudio and Camlibel, M Kanat and Tesi, Pietro},
  journal={IEEE Control Systems Letters},
  volume={4},
  number={3},
  pages={602--607},
  year={2020},
  publisher={IEEE}
}

@article{de2019formulas,
  title={Formulas for data-driven control: Stabilization, optimality, and robustness},
  author={De Persis, Claudio and Tesi, Pietro},
  journal={IEEE Transactions on Automatic Control},
  volume={65},
  number={3},
  pages={909--924},
  year={2019},
  publisher={IEEE}
}

@article{berberich2022combining,
  title={Combining prior knowledge and data for robust controller design},
  author={Berberich, Julian and Scherer, Carsten W and Allg{\"o}wer, Frank},
  journal={IEEE Transactions on Automatic Control},
  volume={68},
  number={8},
  pages={4618--4633},
  year={2022},
  publisher={IEEE}
}

@article{Wang2007,
  author  = {Y. Wang and G. Feng and D. Cheng},
  title   = {Simultaneous stabilization of a set of nonlinear port-controlled Hamiltonian systems},
  journal = {Automatica},
  volume  = {43},
  number  = {3},
  pages   = {403--415},
  year    = {2007}
}

@inproceedings{Dirksz2010,
  author    = {D. A. Dirksz and J. M. Scherpen},
  title     = {Adaptive tracking control of fully actuated port-Hamiltonian mechanical systems},
  booktitle = {IEEE International Conference on Control Applications},
  pages     = {1678--1683},
  year      = {2010}
}

@article{Sprangers2014,
  author  = {O. Sprangers and R. Babu{\v{s}}ka and S. P. Nageshrao and G. A. Lopes},
  title   = {Reinforcement learning for port-Hamiltonian systems},
  journal = {IEEE Transactions on Cybernetics},
  volume  = {45},
  number  = {5},
  pages   = {1017--1027},
  year    = {2014}
}

@article{Fujimoto2003,
  author  = {K. Fujimoto and T. Sugie},
  title   = {Iterative learning control of Hamiltonian systems: I/O based optimal control approach},
  journal = {IEEE Transactions on Automatic Control},
  volume  = {48},
  number  = {10},
  pages   = {1756--1761},
  year    = {2003}
}

@inproceedings{BeckersCDC2023,
  author    = {T. Beckers},
  title     = {Data-driven Bayesian control of port-Hamiltonian systems},
  booktitle = {Proceedings of the 62nd IEEE Conference on Decision and Control (CDC)},
  pages     = {8708--8713},
  year      = {2023}
}

@inproceedings{maithripala2003nonlinear,
  title={Nonlinear dynamic output feedback stabilization of electrostatically actuated MEMS},
  author={Maithripala, DH Sanjeeva and Berg, Jordan M and Dayawansa, Wijesuriya P},
  booktitle={42nd IEEE International Conference on Decision and Control (IEEE Cat. No. 03CH37475)},
  volume={1},
  pages={61--66},
  year={2003},
  organization={IEEE}
}

@article{van2013port,
  title={Port-Hamiltonian systems on graphs},
  author={Van der Schaft, Arjan J and Maschke, Bernhard M},
  journal={SIAM Journal on Control and Optimization},
  volume={51},
  number={2},
  pages={906--937},
  year={2013},
  publisher={SIAM}
}

@inproceedings{BeckersGPPHS2022,
  author    = {T. Beckers and J. Seidman and P. Perdikaris and G. J. Pappas},
  title     = {Gaussian Process Port-Hamiltonian Systems: Bayesian Learning with Physics Prior},
  booktitle = {IEEE Conference on Decision and Control (CDC)},
  year      = {2022}
}

@article{BeckersColombo2025,
  author  = {T. Beckers and L. Colombo},
  title   = {Physics-informed Learning for Passivity-based Tracking Control},
  journal = {arXiv preprint arXiv:2505.01569},
  year    = {2025}
}

@article{Ryalat2018,
  author  = {M. Ryalat and D. S. Laila},
  title   = {A robust IDA-PBC approach for handling uncertainties in underactuated mechanical systems},
  journal = {IEEE Transactions on Automatic Control},
  volume  = {63},
  number  = {10},
  pages   = {3495--3502},
  year    = {2018}
}

\end{document}